\documentclass{amsart}
\usepackage{color}
\usepackage{latexsym,amsfonts,amsthm,amsmath,mathrsfs, amssymb}
\newtheorem{theo}{Theorem}
\newtheorem{lem}{Lemma} 
\newtheorem{prop}{Proposition} 
\newtheorem{defi}{Definition} 
\theoremstyle{definition}

\newcommand{\si}{\sigma}
\newcommand{\Om}{\Omega}
\newcommand{\e}{\epsilon}
\newcommand{\de}{\partial}
\newcommand{\oo}{\omega}
\newcommand{\la}{\lambda}
\newcommand{\R}{\mathbb{R}}

\newcommand{\N}{\mathbb{N}}

\newcommand{\ve}{\varepsilon}
\newcommand{\ds}{\displaystyle}

\numberwithin{equation}{section}

\newcommand{\Rnp}{\ensuremath{\R^{N+1}_{+}}}

\newcommand{\xnp}{\ensuremath{x_{N+1}}}  
\newcommand{\np}[3]{\ensuremath{\left\| #3 \right\|_{#1}^{#2}}} 
\newcommand{\nn}[3]{\ensuremath{\left\| #3 \right\|_{#1}^{#2}}} 
\newcommand{\nt}[3]{\ensuremath{\left| #3 \right|_{#1}^{#2}}}  
\newcommand{\trin}[3]{\ensuremath{\left|\!\left|\!\left| #3 \right|\!\right|\!\right| _{#1}^{#2}}}  
\newcommand{\dual}[3]{\ensuremath{ \left< #2,#3 \right>_{#1} }}

\newcommand{\Cil}{\ensuremath{\mathcal{C}}}
\newcommand{\HuzL}{\ensuremath{H^1_{0,L}(\Cil)}}
\newcommand{\conv}[2]{\ensuremath{ \left< #1,#2 \right> }}

\begin{document}

\title[Fractional Hartree equation without A-R condition]{
The fractional Hartree equation without the Ambrosetti-Rabinowitz condition
}

\author{Mauro Francesconi}
\address{Dipartimento di Matematica e
Informatica\\Universit\`a di Perugia\\Via Vanvitelli 1, 06123
Perugia - Italy}
\email{mauro.francesconi@dmi.unipg.it}

\author{Dimitri Mugnai}
\address{Dipartimento di Matematica e
Informatica\\Universit\`a di Perugia\\Via Vanvitelli 1, 06123
Perugia - Italy}
\email{dimitri.mugnai@unipg.it }




\begin{abstract}
We consider a class of pseudo-relativistic Hartree equations in presence of general nonlinearities not satisfying the Ambrosetti-Rabinowitz condition. Using variational methods based on critical point theory, we show the existence of two non trivial signed solutions, one positive and one negative.
\end{abstract}

\maketitle

\section{Introduction}
In this paper we deal with a general class of pseudo--relativistic Schr{\"o}dinger equations with a Hartree non linearity.
Such equations emerge from the description of pseudorelativistic boson stars (see \cite{Lieb_Thirring} for a physical derivation of the problem), but also as the mean field limit description of a quantum relativistic Bose gas (see \cite{ES} and \cite{LY}).
Fr{\"o}hlich and Lenzmann in \cite{Frohlich_Lenzmann} and \cite{Frohlich_Jonsson_Lenzmann} approached the problems of existence, blowing up and stability of solutions. The problem they studied in \cite{Frohlich_Lenzmann} took the following form:
\begin{equation}\label{1}
i \psi_t = \sqrt{- \Delta + m^2} \psi   -\left( \frac{1}{|x|}* |\psi|^2 \right)\psi  \mbox{ in }\R^3,
\end{equation}

Here $\psi$ is a complex valued wave function which describes the quantum status of a particle, while the operator involving the square root represents its relativistic kinetic and rest energies, and reduces to the usual half Laplacian $(-\Delta) ^{1/2}$ when $m=0$. Besides, the term $1/|x|$ inside the convolution product stands for the Newtonian gravitational potential in $\R^3$ and represents repulsive forces among the particles.



In \cite{Mugnai_Pseudo}, a generalized version of \eqref{1} is studied, allowing for an additional potential term $f:\R^N \times \R \rightarrow \R$, which takes into account other external forces, and, in addition, a general field potential $W$ replaces the Newtonian one. In this setting, equation \eqref{1} takes the following form:
\begin{equation}\label{22}
i \psi_t = \sqrt{- \Delta + m^2} \psi   -\la \left( W* |\psi|^2 \right)\psi - f(x,\psi) \mbox{ in }\R^N,
\end{equation}
with $\la \in \R$.

In this paper we search solutions of a problem corresponding to \eqref{22} but settled in a bounded domain $\Omega$ of $\R^N$.
This allows us to remove the hypothesis of radial symmetry of the solutions and of the potential $f$ assumed in \cite{Mugnai_Pseudo}.
The problem we study is the following one:
\begin{align}\label{2}
\left\{
\begin{array}{llr}
i \psi_t(x,t) = & \sqrt{- \Delta + m^2} \psi(x,t)  - &\\
  & \ds-\la \left( \int_{\Omega} G(x,y) |\psi(y,t)|^2 dy \right)\psi(x,t) - f(x,\psi(x,t)) & \mbox{ in }\Omega, \\
 \psi (x,t) = & 0  &\mbox{ on }\partial \Omega, \;\; \forall t 
\end{array}
\right.
\end{align}
with $\la \in \R$ and $\Omega \subset \R^N$ bounded.

In \eqref{2}, passing from $\R^N$ to $\Omega$, we have replaced the Newtonian--like kernel $W(x-y)$ with the Green function $G(x,y)$ associated to the Laplace operator in $\Omega$ (indeed this is the Coulomb--type interaction between particles in boson stars), and we consider homogeneous boundary conditions on $\partial \Omega$. In this way, the corresponding potential $\phi$ at time $t$ takes the form $\ds \phi(x) = \int_{\Omega} G(x,y) |\psi(y,t)|^2 dy $. From now on we will adopt the symbol $\ds\conv{G}{\psi} = \int_{\Omega} G(x,y) |\psi(y,t)|^2 dy$ for the previous potential term. We also note that
potential $\ds \phi(x) = \conv{G}{\psi} $ is the solution of the linear problem
\begin{align}\label{Green}
\left\{
\begin{array}{l}
- \Delta \phi(x,t) = 4 \pi |\psi(x,t)|^2 \mbox{ in } \Omega\\
 \phi = 0 \mbox{ on }\partial \Omega
 \end{array}
 \right.
\end{align}
for every $t$, so that problem \eqref{2} can be written as a system with an additional equation for $\phi$, as similarly done in \cite{bf}, \cite{teadim}, \cite{tdnonex}, \cite{KGBI}.

It is worth reminding some general properties of Green functions for $C^1$ bounded domains $\Omega$, which we shall use later (for instance, see \cite{Gunther_Widman}).
Green functions $G:\Omega \times \Omega \rightarrow \R \cup \{ \infty\}$ are non negative, symmetric with respect to their variables, and when $N\geq 3$ they verify the inequality $G(x,y) \leq C |x-y|^{2-N}$, where the right hand side is the kernel of the Newtonian like potential.
More generally,  inspired by the above inequality, in our setting we consider a function $G$ which is symmetric, non negative and satisfies certain integrability conditions.
To be precise, we will require $G(x,y) \leq W(x-y)$, with $W$ satisfying some integrability conditions which cover the case of the Newtonian kernel for $N \geq 3$.

In order to obtain existence of solutions for problem \eqref{2}, it is crucial to specify some hypothesis on the external potential $\ds F(x,\psi) =\int_0^{\psi} f(x,s) ds$. 
The prototype for $F$ is a power--like potential, so it is natural to require $F(x,s)=F(x,|s|)$ and $f(x,e^{i\theta}|s|) = 
e^{i\theta}f(x,|s|)$. This is not restrictive in the setting of Abelian Gauge Theories (see \cite{Benci_Fortunato}, \cite{Mugnai_Solitary}), and it allows us to search for real solutions of the stationary equation associated to \eqref{2}.
Indeed, we focus on solutions in the form of solitary waves, i.e. on functions of the form
\begin{equation}\label{solitary}
\psi(x,t) = e^{-i \oo t} u(x),
\end{equation}
where $\oo \in \R$ and $u:\Om \rightarrow \R$.

After substitution of \eqref{solitary} into \eqref{2}, and considering that the operator $\sqrt{- \Delta + m^2}$ acts only on the spatial coordinates, we see that function $u$ satisfies the following stationary equation:
\begin{align}\label{stationary}
\left\{
\begin{array}{lr}
\sqrt{- \Delta + m^2} u - \oo u -\la \conv{G}{u^2} u - f(x,u) = 0 &\mbox{ in }\Om. \\
u=0 &\mbox{ on } \partial \Omega
\end{array}
\right.
\end{align}

We remark that a standard assumption on $F$ for solving stationary equations like \eqref{stationary} is the fulfillment of the usual Ambrosetti-Rabinowitz condition, see \cite{Benci_Fortunato}, \cite{Mugnai_Solitary}, \cite{Mugnai_Poisson} (or a reversed one (see \cite{Mugnai_Pseudo}). We recall that this condition reads as follows: there exists $\mu>2$ such that
\begin{equation}\label{A-R}
0 < \mu F(x,s) \leq sf(x,s) \ \mbox{ for a.e. }x \in \Omega \mbox{ and for all } s \in \R.
\end{equation}
Condition \eqref{A-R} is very useful to prove that Palais--Smale sequences are bounded, and in turn that Palais Smale condition (PS-condition in short) holds, so that an essential ingredient in variational methods is guaranteed.

Although very convenient, condition \eqref{A-R} rules out many interesting non linearities.
For this reason, many efforts have been done recently to remove or relax it (see Li, Wang, Zeng \cite{Li_Wang_Zeng}, \cite{Li_Zeng}, \cite{Liu}, \cite{Li_Wang}, Myagaki and Souto \cite{Miyagaki_Souto}). In this paper, we adopt the strategy of Mugnai and Papageorgiou (\cite{Mugnai_Papageorgiou}),
which consists in requiring a quasi-monotone property for the function
\begin{equation}\label{sigma_fun}
\sigma(x,s) = f(x,s)s - 2F(x,s) \mbox{ for a.e. }x \in \R^N \mbox{ and for all } s \in \R,
\end{equation}
see Section \ref{sec2} for the precise assumptions. The purpose of this paper is to show that, under our weak assumptions, equation \eqref{stationary} has two non trivial bounded solutions. In order to prove this result, we will employ an {\em a priori} estimate for solutions of \eqref{stationary} of independent interest, see Proposition \ref{Regularity_of_weak_solutions} below.

\section{Extended problem and assumptions}\label{sec2}

We start this section by reviewing the essential tools to face problem \eqref{stationary}. We follow the idea of extending equation \eqref{stationary} to an equivalent one in higher dimension, by means of the Dirichlet-to-Neumann operator $\ds \left.-\frac{\partial }{\partial x_{N+1}}\right|_{x_{N+1}=0}$ (see  \cite{Cabre_Tan} for this procedure in bounded domains and \cite{Caffarelli_Silvestre} for the whole spatial domain). This method leads us to consider the following problem:
\begin{equation}\label{equiv}
\begin{cases}
- \Delta v+ m^2 v =0  & \mbox{ in }\Cil,\\
\ds-\frac{\partial v}{\partial x_{N+1}} =  \oo v +\la  \conv{G}{v^2} v + f(x,v) &\mbox{ on } \Om \times \{0\},\\
v = 0 &\mbox{ on } \partial_L \Cil := \partial \Om \times [0,\infty). 
\end{cases}
\end{equation}
where $\Cil = \Om \times (0,\infty)$ is the positive half cylinder with base $\Om$ and $\partial_L \Cil$ is its lateral boundary. As in \cite{Cabre_Tan}, we have that if $v$ satisfies \eqref{equiv}, then its trace $u(\cdot):=v(\cdot,0)$ on $\Omega \times \{0\}$ satisfies problem \eqref{stationary}.

In order to define a weak solution of \eqref{equiv} (see Cabr{\'e} and Tan \cite{Cabre_Tan}), we introduce the Sobolev space  
\begin{equation}\label{HuzL}
\HuzL = \Big\{v \in H^1(\Cil) \,: \, v=0 \mbox{ a.e. on } \partial_L \Cil  \Big\},
\end{equation}
equipped with the inner product
\[
\langle u,v\rangle = \ds \int_{\Cil} (Du\cdot Dv + m^2 uv)dxdx_{N+1},
\]
which makes $\HuzL$ a Hilbert space with respect to the induced norm (here and henceforth we shall denoted by $x$ a general point of $\Omega$ and by $x_{N+1}$ an element of $[0,\infty)$).

Hence, we can give the following
\begin{defi}[Weak Solution]
A function $v \in \HuzL$ is a weak solution of problem \eqref{equiv} if
\begin{equation}\label{weak_sol}
\int_{\Cil} \left[ Dv\cdot Dw + m^2  v w  \right]dxdx_{N+1}= \int_{\Omega} \left[ \oo v + \la \conv{G}{v^2}v  + f(x,v) \right] w\,dx
\end{equation}
for every $w \in \HuzL$.
\end{defi}

Recalling that $G$ is symmetric, it is easy to see that problem \eqref{weak_sol} is of variational nature, so that a function $v$ in $ \HuzL$ satisfies \eqref{weak_sol} if and only if it is a critical point of the following energy functional $J: \HuzL\to \R$ defined as
\begin{equation}\label{funct}
J(v) = \frac{1}{2} \int_{\Cil} \left[ |Dv|^2 + m^2  v^2  \right] dxdx_{N+1}- \int_{\Omega} \left[ \frac{\oo}{2} v^2 + \frac{\la}{4} \conv{G}{v^2}v^2  + F(x,v) \right]dx,
\end{equation}
or, in compact form,
\begin{equation}\label{funct_norm}
J(v) = \frac{1}{2}  \np{2}{2}{Dv} + \frac{m^2}{2}  \np{2}{2}{v}  - \frac{\oo}{2}  |v|_2^2 - \frac{\la}{4} \int_{\Omega} \conv{G}{v^2}v^2 dx - \int_{\Omega}  F(x,v)dx,
\end{equation}
where $\nt{p}{}{\cdot}$ and $\np{p}{}{\cdot}$ denote the $L^p$ norm in $\Omega$ and in $\Cil$, respectively. The derivative of functional J acts on any function $w \in \HuzL$ in the following way:
\begin{equation}\label{funct_derivative}
\begin{aligned}
J'(v)w &= \dual{2,N+1}{Dv}{Dw} + m^2 \dual{2,N+1}{v}{w}   -\oo  \dual{2,N}{v}{w} \\
&- \lambda\int_{\Omega} \conv{G}{v^2} v w\,dx  - \int_{\Omega}  f(x,v)w\,dx,
\end{aligned}
\end{equation}
where $\dual{p,n}{\cdot}{\cdot}$ denotes the duality product in $[L^p(\mathcal{C})]'\times L^p(\mathcal{C})$ when $n=N+1$ or in $[L^p(\Omega)]'\times L^p(\Omega)$ when $n=N$.

Of course, the previous considerations are just formal ones, if we don't assume appropriate conditions on $f$ and $G$. For this, throughout this paper we make the following assumptions, which guarantee that the formal considerations above are indeed true:
\medskip

\noindent $\Omega\subset \R^N$, $N\geq2$, is a bounded domain with boundary of class $C^{2,\alpha}$;\\ 
{\bf (H)}: $f:\Omega \times \R \rightarrow \R$ is a Carath\'eodory function, with $f(x,0)=0$ for a.e. $x \in \Omega$. Moreover, if $\ds F(x,s): = \int_0^s f(x,\tau) d\tau$, we suppose that
\begin{description}
\item[Hi)] there exist $\ds c>0$, $\ds a \in L^{\infty}(\Omega)$, $a\geq 0$ a.e. in $\Omega$, and 
$\ds r \in \left(2, \frac{2N}{N-1}\right)$ such that $$|f(x,s)|\leq a(x) + c|s|^{r-1}$$ for a.e. $x \in \Omega$ and all $s\in \R$; \\

\item[Hii)] $\ds \lim_{|s|\rightarrow  \infty} \frac{F(x,s)}{s^2} = +\infty$ uniformly for a.e. $x \in \Omega$;

\item[Hiii)] if $\ds \sigma(x,s) := f(x,s)s - 2F(x,s)$, there exists $\beta^* \in L^1_+(\Omega)$ s.t.
\[
\sigma(x,s) \leq \sigma(x,t) + \beta^*(x) 
\]
for a.e. $\ds x \in \Omega$ and all $0\leq s\leq t$ or $t\leq s\leq 0$;			
\item[Hiv)] there exists $\theta \in L^{\infty}_+ (\Omega)$ with $\theta_{\infty} = |\theta|_{\infty} < m - \oo$ such that $$\ds \limsup_{s\rightarrow  0} \frac{F(x,s)}{s^2} \leq \frac{\theta(x)}{2}.$$
\end{description}

Concerning $G$, as a generalization of the Green function of the domain $\Omega$, which belongs to $L^r(\Omega)$ for $r<N$, we assume the natural hypothesis\\
 (\textbf{G}): $G\geq 0$, $G(x,y)=G(y,x)$ and $G(x,y) \leq W(x-y)$ for every $(x,y)\in \Omega \times \Omega$, where $W\geq 0$ in $\R^N$, $W\in L^r(\R^N)$ for some $r\in (\frac{N}{2},\infty)$ and $W=0$ in $\R^N\setminus \Omega$.

\section{Technical inequalities}

We now establish some useful inequalities which will be used extensively throughout the paper.

We start with the continuous inclusions (see \cite[Lemma 2.4]{Cabre_Tan})
\begin{equation}\label{inclusion}
\HuzL \hookrightarrow L^r(\Omega) \quad \mbox{for all $r \in \left[1,\frac{2N}{N-1}\right]$},
\end{equation}
and the compact ones (see \cite[Lemma 2.5]{Cabre_Tan})
\begin{equation}\label{compact_inclusion}
\HuzL \hookrightarrow L^r(\Omega) \quad \mbox{for all $r \in \left[1,\frac{2N}{N-1}\right)$}.
\end{equation}

Now, take $v\in \mathcal{C}^{\infty}(\Rnp)\cap \HuzL$; then, proceeding as in \cite{Mugnai_Pseudo},
\begin{multline}\label{trace_in_1}
\int_{\Omega} |v(x,0)|^q dx = \int_{\Omega} \left( - \int_0^{\infty} \frac{\de}{\de x_{N+1}} |v(x,x_{N+1})|^q d \xnp \right) d x \\
= -q\int_{\Cil} |v(x,x_{N+1})|^{q-2} v(x,x_{N+1}) \frac{\de v}{\de x_{N+1}} (x,x_{N+1})  d x \; d\xnp 
\end{multline}
Applying the H\"{o}lder inequality with exponent 2, we get
\begin{equation}\label{trace_in_2}
\int_{\Omega} |v(x,0)|^q dx  \leq q  \np{2}{}{v^{q-1}}  \np{2}{}{\frac{\de v}{\de \xnp}}  \leq q  \np{2(q-1)}{q-1}{v}  \np{2}{}{Dv}.
\end{equation}
By interpolation for $2(q-1)$ between 2 and $\ds 2^\sharp=\frac{2N}{N-1}$, followed by the Sobolev embedding inequality, we find the trace inequality
\[
|v|_q\leq S_q\|v\|
\]
for every $v\in  \mathcal{C}^{\infty}(\Rnp)\cap \HuzL$, where $S_q$ is an absolute positive constant.

Moreover, if we use the Cauchy inequality in \eqref{trace_in_1}, when $q=2$ we obtain
\begin{equation}\label{trace_in_3}
|v|_2^2 = \int_{\Omega} |v(x,0)|^2 dx  \leq   \e \np{2}{2}{v} + \frac{1}{\e} \np{2}{2}{\frac{\de v}{\de \xnp}}.
\end{equation}

In particular, choosing $\e=m$, (\ref{trace_in_3}) gives the following estimate for the trace norm:
\begin{equation}\label{tracem}
|v|_2^2\leq m\np{2}{2}{v}+ \frac{1}{m} \np{2}{2}{\frac{\de v}{\de \xnp}}\leq m\np{2}{2}{v}+ \frac{1}{m} \np{2}{2}{Dv} \quad \forall \; v \in  \mathcal{C}^{\infty}(\Rnp)\cap \HuzL.
\end{equation}
Finally,  by density, we have that all the inequalities above hold for every $v \in \HuzL$.
\smallskip

Now, we proceed by inferring some inequalities on $F$ which come directly from hypothesis {\bf (H)}. First, a direct integration of {\bf Hi)} gives
\begin{equation}\label{Hi_consequence}
|F(x,s)| \leq a(x)|s| + \frac{c}{r}|s|^r \quad \mbox{ for a.e. $x\in \Omega$ and for all }s\in\R.
\end{equation}

Furthermore, from {\bf Hiv)} we can say that for any $\e>0$ there exists $\delta=\delta(\e)>0$ such that
\begin{equation}\label{Hiv_consequence}
F(x,s) \leq \frac{\theta(x) + \e}{2}s^2\quad \mbox{ for a.e. $x\in \Omega$ and for all } |s| < \delta.
\end{equation}

From \eqref{Hi_consequence} and \eqref{Hiv_consequence}, we deduce  that
\begin{equation}\label{Hi_&_Hiv_joined}
F(x,s) \leq \frac{\theta(x) + \e}{2}s^2 + C_\ve |s|^r \quad \mbox{for a.e $x \in \Omega$ and for all }s\in \R,
\end{equation}
where $C_\ve=C_{\delta(\ve)}=\frac{\|a\|_\infty}{\delta^{r-1}}+\frac{c}{r}$.
\smallskip

We end this section by showing an estimate involving field potential $G$. First, extend $G$ outside $\Omega \times \Omega$ and any function $u,\, v,\, w\in \HuzL$ outside $\Omega\times \{0\}$ in a trivial way. By H\"older's inequality, if $2q\in [2,2^\sharp]$, we get
\[
\left| \int_\Omega \langle G,v^2\rangle uw\,dx \right|\leq |\langle W, v^2\rangle|_{q'}
 |uw|_{q}=|W\ast v^2|_{q'}|uw|_q,
\]

where $\ast$ denotes the usual convolution product in $\R^N$ and where we have denoted traces on $\Omega\times \{0\}$ simply by functions themselves. Now, apply Young's inequality for convolutions, choosing $q$ so that
$\frac{1}{q'}=\frac{1}{r}+\frac{1}{q}-1$, that is $q=\frac{2r}{2r-1}$, so that from the previous inequality we get
\begin{equation}\label{W_inequality}
\left| \int_\Omega \langle G,v^2\rangle uw\,dx \right|\leq|W|_r|v|_{2q}^2|uw|_{q}\leq |W|_r|v|_{2q}^2|u|_{2q}|w|_{2q}.
\end{equation}
We remark that, since $r\in\left(\frac{N}{2},\infty\right)$, we have $1< q<N/(N-1)$. Finally, by the
interpolation and the Sobolev inequalities, we get that there exists
$C_G>0$ such that
\begin{equation}\label{maggconv}
\left|\int_{\Omega}\langle G, v^2\rangle uw\,dx\right| \leq C_G\|v\|^2\|u\|\|w\| \quad \mbox{
for any }u,\,v,\,w\in \HuzL.
\end{equation}

\section{Regularity of weak solutions}

In this section we briefly complement Cabr\'e--Tan's results on regularity of weak solutions: such results seem to be very natural, and are related to the regularity properties established in \cite[Proposition 3.1]{Cabre_Tan} for $m=0$ and in \cite[Theorem 3.2 and Proposition 3.9]{Zelati_Nolasco}.

\begin{prop}\label{Regularity_of_weak_solutions}
Suppose $\Omega \subset \R^N$ is a Lipschitz bounded domain with $\alpha \in (0,1)$. Then, under hypotheses {\bf(H)} and {\bf(G)}, all weak solutions $v$ of problem \eqref{equiv} are of class $L^{\infty}(\Cil)\cap C^{\alpha}(\overline{\Omega})$ and $u \in L^\infty(\Omega)$, being $u(\cdot)=v(\cdot,0)$ a solution of \eqref{stationary}. Moreover, for every $p\in[1,\infty]$ there exists $M_p>0$ such that
\[
\|v\|_{L^p(\mathcal{C})}\leq M_p
\]
and also
\[
\|u\|_{L^p(\Omega)}\leq M_p.
\]
\end{prop}
\begin{proof}
In order to prove the first statement, we need only minor changes in the proof of \cite[Theorem 3.2]{Zelati_Nolasco}, and for this here we will be sketchy. As usual, set $v_T = \max{\left\{v_+,T\right\}} $ and, fixed $\beta>0$, apply \eqref{weak_sol} with $w=v v_T^{2\beta} \in \HuzL$. Then we get
\[
\begin{aligned}
\|vv_T^\beta\|^2&=\int_{\Cil}\left(|D(vv_T^\beta)|^2+(vv_T^\beta)^2\right)\,dxdx_{N+1}\\
&\leq c_\beta\int_\Omega \left(\omega v^2v_T^{2\beta}+\lambda \langle G,v^2\rangle v^2v_T^{2\beta}+f(x,v)vv_T^{2\beta} \right)dx.
\end{aligned}
\]
By {\bf(G)} and {\bf(Hi)} we can easily recover estimate (3.3) of \cite{Zelati_Nolasco}, obtaining
\[
\int_{\Cil}\left(|D(v_+^{\beta+1})|^2+(v_+^{\beta+1})^2\right)\,dxdx_{N+1}\leq c_\beta \int_\Omega\left( c v_+^{2\beta+2}dx+gv_+^{2\beta+2}\right)dx
\]
for some $g\in L^N(\Omega)$. Hence, proceed as in \cite{Zelati_Nolasco} to obtain the claim.
\end{proof}

Now, in order to state the regularity results, let us consider the problem
\begin{align}\label{pg}
\left\{
\begin{array}{lr}
\sqrt{- \Delta + m^2} u=g(x) &\mbox{ in }\Om. \\
u=0 &\mbox{ on } \partial \Omega,
\end{array}
\right.
\end{align}
and its related extended one
\begin{equation}\label{pg1}
\begin{cases}
- \Delta v+ m^2 v =0  & \mbox{ in }\Cil,\\
-\frac{\partial v}{\partial x_{N+1}} =  g(x) &\mbox{ on } \Om \times \{0\}\\
v = 0 &\mbox{ on } \partial_L \Cil := \partial \Om \times [0,\infty). 
\end{cases}
\end{equation}

From now on, if $v$ solves \eqref{pg1}, the associated solution of \eqref{pg} will be denoted by $u={\rm tr}\,u$, meaning that $u(\cdot)=v(\cdot,0)$. Adapting the proof of  \cite[Proposition 3.1]{Cabre_Tan}, and representing the space of all traces on $\Omega\times \{0 \}$ of functions in $\HuzL$ by the symbol $ {\mathcal V}_0$, we immediately have the following proposition:
\begin{prop}\label{stimaLieb}
Let $\alpha\in (0, 1)$, $\Omega$ be a $C^{2,\alpha}$ bounded domain of $\R^N$, $v\in \HuzL$ be the weak solution of \eqref{pg1}, $u={\rm tr}\,v$ be the weak solution of \eqref{pg} and $g\in {\mathcal V}_0^\ast\cap L^p(\Omega)$ for some $p\in(1,\infty)$. Then $v\in W^{2,p}(\Omega\times (0,R))$ for all $R>0$. If $g\in C^{\alpha}(\overline{\Omega}\times \R)$ and $g|_{\partial \Omega}\equiv 0$, then $v\in C^{1,\alpha}(\overline{\mathcal{C}})$, $u\in C^{1,\alpha}(\overline{\Omega})$.
\end{prop}

\section{Constant sign solutions}

Our main result is the following
\begin{theo}\label{theorem_2_Mountain_Pass_solutions}
Under hypotheses {\bf(H)} and {\bf(G)}, for any $\lambda>0$ problem \eqref{equiv} admits two non trivial bounded solutions, one strictly positive and one strictly negative in $\Omega$.
\end{theo}

The proof of Theorem \ref{theorem_2_Mountain_Pass_solutions} is based on an application of the Mountain Pass Theorem to functionals $J_+(v)$, $J_-(v)$, defined in $\HuzL$ as follows:
\begin{equation}\label{funct_pm_W}
\begin{aligned}
J_{\pm}(v) &= \frac{1}{2} \int_{\Cil} \left[ |Dv|^2 + m^2  v^2  \right]dxdx_{N+1}\\
&- \int_{\Omega} \left[ \frac{\oo}{2} v^2 + \frac{\la}{4} \conv{G}{{v^\pm}^2}{v^\pm}^2  
+ F_{\pm}(x,v) \right]dx.
\end{aligned}
\end{equation}
Here $\ds F_{\pm}(x,v)= F(x,\pm v^{\pm})$, where $v^+=\max\{v,0\}$ and $v^-=\max\{-v,0\}$ denote the positive and the negative part of $v$, respectively.

However, though verifying the geometrical assumptions of the mountain pass is not very hard, thanks to some inequalities established above, the verification of the compactness condition is the hardest part. Moreover, since our assumptions are very general and do not imply a growth of order $q>2$ at infinity, the usual Palais--Smale condition has to be replaced by the generally weaker Cerami condition:
\begin{defi}
Let $X$ be a Banach space with topological dual $X^\ast$. A $C^1$ functional $J:X\to \R$ is said to satisfy the Cerami condition - (C) for short - if every sequence $(u_n)_n\subset X$ such that $(J(u_n))_n$ is bounded and $(1+\|u_n\|)J'(u_n)\to 0$ in $X^\ast$ as $n\to \infty$, has a convergent subsequence.
\end{defi}
As shown by Bartolo-Benci-Fortunato \cite{bbf}, such a condition can successfully replace the Palais--Smale condition in proving a Deformation Theorem, and consequently a minimax theory for critical values. In particular, the classical Mountain Pass Theorem holds under this compactness condition and we will apply such a version (see \cite[Corollary 5.2.7]{gp}).

Now, we will check the Mountain Pass hypothesis for $J_{+}$, but analogous results hold true with minor changes for $J_-$ (and $J$). In this way we will exhibit the existence of a positive and of a negative solution.

First, we show that $J_+$ has a strict local minimum at the origin: take $v\in\HuzL$, then, by \eqref{maggconv} and \eqref{Hi_&_Hiv_joined} we find
\begin{equation}\label{strict_pos_1}
\begin{aligned}
J_+(v) &\geq \frac{1}{2} \np{2}{2}{Dv} + \frac{m^2}{2}  \np{2}{2}{v}  - \frac{\oo}{2}  |v|_2^2 - \lambda C_G \np{}{4}{v} \\
&
- \int_\Omega \frac{\theta(x) + \e}{2}(v^+)^2dx - C_{\e} |v^+|_r^r\\
&\geq \frac{1}{2}  \left[  \np{2}{2}{Dv} + m^2  \np{2}{2}{v} - (\oo + \theta_{\infty} + \e) |v|^2_2 \right]  - \lambda C_G\np{}{4}{v} - \tilde{C_\ve}\np{}{r}{v},
\end{aligned}
\end{equation}
for some $\tilde{C_\ve}>0$.

If $\omega+ \theta_{\infty}+\e\leq0$, the claim follows immediately, since $r>2$. If $\omega+ \theta_{\infty}+\e>0$, by {\bf Hiv)} we can suppose that $\theta_\infty+\e<m-\omega$, and using \eqref{tracem}, we find a positive constant $\tilde{c}$ such that
\begin{equation}\label{sotto52}
\begin{aligned}
J_+(v)&\geq \frac{1}{2}\left(1-\frac{\omega+\theta_\infty+\e}{m}\right)\|Du\|_2^2+\frac{1}{2} \Big(m^2-m(\omega+\theta_\infty+\e)\Big)\|v\|_2^2\\
&- C_G\np{}{4}{v} - C_\ve\np{}{r}{v}\\
&\geq \tilde{c} \np{}{2}{v}- C_G\np{}{4}{v} - C_\ve\np{}{r}{v}.
\end{aligned}
\end{equation}
Thus,  0 is a strict local minimum for $J_+$, and there exists $\rho>0$ such that
\begin{equation}\label{eta_+}
0=J_+(0) < \inf\{J_+(u)\,:\, \nn{}{}{u}= \rho \}:= \eta_+.
\end{equation}

Next, by \eqref{Hi_&_Hiv_joined}, we have that for every $v\in \HuzL$


\begin{equation}\label{funct_norm_+_W}
\begin{aligned}
J_+(v) 
& \leq \frac{\max{ \{1,m^2 \} }}{2}\np{}{2}{v} - \frac{\oo}{2}  |v|_2^2 - \frac{\la}{4} \int_{\Omega} \conv{G}{(v^+)^2}{v^+}^2dx\\
& + \frac{\theta_\infty+\ve}{2}|v|_2^2+C_\ve|v|_r^r 
\end{aligned}
\end{equation}
Thus, if $t>0$, and we choose a nonnegative $v \in \HuzL$, 
there exist positive constants $c_1$, $c_2$, $c_3$, $c_4$, $c_5$ such that
\begin{equation}\label{funct_norm_+_decreasing}
J_+(tv) 
\leq c_1 t^2 - c_2 t^2 - c_3 t^4 + c_4 t^2 + c_5 t^r \xrightarrow[t \rightarrow \infty]{}  - \infty,
\end{equation}
since $r<4$ for all $N\geq2$.

In this way we have proved the validity of the geometric conditions of the Mountain Pass Theorem. Next, we proceed by showing the compactness hypothesis in the form of the Cerami condition.

\subsection{Verification of (C)}
Let $(u_n)_n$ be a Cerami sequence in $\HuzL$, i.e. a sequence such that
\begin{eqnarray}\label{Cerami_Condition}
\left\{
\begin{array}{l}
|J_+(u_n)|\leq M \; \forall \,n \in \N \mbox{ and}\\
\left( 1 + \np{}{}{u_n} \right) J_+'(u_n) \xrightarrow[n \rightarrow \infty]{} 0 \mbox{ in } \left[\HuzL\right]^*
\end{array}
\right.
\end{eqnarray}
for some $M>0$. Now, we prove that $(u_n)_n$ admits a converging subsequence.

\begin{lem}
The sequence $(u_n)_n$ is bounded.
\end{lem}
\begin{proof}
From \eqref{Cerami_Condition} we have
\begin{equation}\label{Cerami_2}
|J_+'(u_n)h| \leq \frac{\e_n \np{}{}{h}}{1+\np{}{}{u_n}}\; \mbox{ for all }  h \in \HuzL,
\end{equation}
where $\e_n\to 0$ as $n\to \infty$. Writing \eqref{Cerami_2} explicitly, using the analogue of \eqref{funct_derivative} for $J_+$, we find 
\begin{equation}\label{eq_h_W}
\begin{aligned}
\left|\int_\Cil [Du_n \cdot Dh \right. &+ m^2 u_nh]dxdx_{N+1} - \oo \int_\Omega u_nh\,dx\\
&\left. - \la \int_{\Omega}\conv{G}{(u_n^+)^2} u_n^+ h \,dx- \int_{\Omega} f_+(x,u_n)h\,dx \right| 
\leq \frac{\e_n\np{}{}{h}}{1+\np{}{}{u_n}} 
\end{aligned}
\end{equation}
Now, in \eqref{eq_h_W} we choose alternatively $h= u_n^-$  and $h= u_n^+$, so that  we find, respectively,
\begin{equation}\label{C_with_un-}
\left| \np{2}{2}{Du_n^{-}} + m^2 \np{2}{2}{u_n^{-}} -\oo \nt{2}{2}{u_n^{-}}  \right| 
\leq \frac{\e_n\np{}{}{u_n^{-}}}{1+\np{}{}{u_n}},
\end{equation}
and
\begin{equation}
\begin{aligned}\label{C_with_un+} 
\bigg| \np{2}{2}{Du_n^{+}} \bigg. &+ m^2 \np{2}{2}{u_n^+} -\oo \nt{2}{2}{u_n^+}\\
&\left.- \la \int_{\Omega}\conv{G}{(u_n^+)^2}(u_n^{+})^2dx  -\int_{\Omega} f(x,u_n^{+})u_n^{+} dx\right|
\leq\frac{\e_n\np{}{}{u_n^{+}}}{1+\np{}{}{u_n}}.
\end{aligned}
\end{equation}
By using \eqref{tracem}, from \eqref{C_with_un-} we immediately see that $\ds (u_n^{-})_n$ is bounded in $\HuzL$.

Now, we rewrite $J_+(u_n)$ as sum of two components, acting on $u_n^+$ and $u_n^-$ separately:
\begin{equation}\label{C_bound_1_W}
\begin{aligned}
J_+(u_n) &= 
\frac{1}{2}\left[ \np{2}{2}{Du_n^{-}} + m^2 \np{2}{2}{u_n^{-}} -\oo \nt{2}{2}{u_n^{-}}\right]   \\
&+\frac{1}{2}\left[ \np{2}{2}{Du_n^{+}} + m^2 \np{2}{2}{u_n^{+}} -\oo \nt{2}{2}{u_n^{+}}\right]\\
& - \frac{\la}{4} \int_{\Omega}\conv{G}{(u_n^+)^2}(u_n^{+})^2dx  -\int_{\Omega} F(x,u_n^{+})\,dx.
\end{aligned}
\end{equation}
By \eqref{C_bound_1_W} we can write
\begin{equation}\label{W_term_limited_1_W}
\begin{aligned}
J_+(u_n)& = \frac{1}{2} J_+'(u_n)(u_n^{-}) +\frac{1}{2} J_+'(u_n)(u_n^{+}) + \frac{\la}{4} \int_{\Omega}\conv{G}{(u_n^+)^2}{(u_n^+)^2}dx \\
&+ \frac{1}{2} \int_{\Omega} f(x,u_n^{+})u_n^{+} dx- \int_{\Omega} F(x,u_n^{+})dx \\
&= \frac{1}{2} J_+'(u_n)(u_n^{-}) + \frac{1}{2} J_+'(u_n)(u_n^{+}) + \frac{\la}{4} \int_{\Omega}\conv{G}{(u_n^+)^2}{(u_n^+)^2} dx\\
&+ \frac{1}{2} \int_{\Omega} \si(x,u_n^{+})\,dx.
\end{aligned}
\end{equation}
The first two terms of the last side are bounded by $\e_n$, see (\ref{Cerami_2}). Then,
\[
J_+(u_n) \geq 
-\e_n + \frac{\la}{4} \int_{\Omega}\conv{G}{(u_n^+)^2}{(u_n^+)^2}dx + \frac{1}{2} \int_{\Omega} \si(x,u_n^{+})\,dx.
\]
In addition, the last term is limited from below, since condition {\bf Hiii)} implies that
\[
0 = \si(x,0) \leq \si(x,t) + \beta^*(x) \; \forall \,t \geq 0.
\]
As a consequence, from (\ref{W_term_limited_1_W}) we get
\[
J_+(u_n) \geq 
-\e_n + \frac{\la}{4} \int_{\Omega}\conv{G}{(u_n^+)^2}{(u_n^+)^2}dx - \frac{1}{2} |\beta^*|_1
\]
Finally, from the bound on $J_+(u_n)$, see (\ref{Cerami_Condition}, and the non negativity of $G$, we get that there exists $M\geq 0$ such that
\begin{equation}\label{W_term_limited_5_W}
 \int_{\Omega}\conv{G}{(u_n^+)^2}(u_n^+)^2dx \in [0,M] \; \; \forall \,n \in\N.
\end{equation}



Let us now remark that, by \eqref{tracem},
\[
\np{2}{2}{D\cdot} + m^2 \np{2}{2}{\cdot} -\oo \nt{2}{2}{\cdot}
\]
defines a quantity which is equivalent to $\|\cdot\|^2$ in $\HuzL$, which we shall denote by $\trin{}{2}{\cdot}$, from now on. Using this notation, starting from \eqref{C_with_un+}, by using \eqref{W_term_limited_5_W}
, we get the existence of $M_2\geq0$ such that
\begin{equation}\label{for_contradiction_1}
 \left| \trin{}{2}{u_n^{+}} - \int_{\Omega} f(x,u_n^{+})u_n^{+} \right| \leq M_2 \mbox{ for all } n \in\N.
\end{equation}

Moreover, from the bound on $J_{+}(u_n)$ given by \eqref{Cerami_Condition}, from the bound on the Green--potential term $G$ in \eqref{W_term_limited_5_W} and from the bound on $u_n^{-}$, starting from \eqref{C_bound_1_W}, we also get
\begin{equation}\label{for_contradiction_2}
 \left| \frac{1}{2}\trin{}{2}{u_n^{+}} - \int_{\Omega} F(x,u_n^{+}) \right| \leq M_3 \mbox{ for all } n \in\N
\end{equation}
for some $M_3\geq0$. Combining both (\ref{for_contradiction_1}) and (\ref{for_contradiction_2}) we get
\begin{equation}\label{for_contradiction_3}
 \left| \int_{\Omega} \sigma(x,u_n^{+}) \right| \leq M_4 \mbox{ for all } n \in\N
\end{equation}
for some $M_4\geq0$.

\subsubsection*{Claim: $u_n^{+}$ is bounded}
We prove it by contradiction. Suppose that $(u_n^{+})_n$ is not bounded; then, we may assume that $\np{}{}{u_n^{+}} \xrightarrow[n]{} \infty$, and that
\[
y_n := \frac{u_n^{+}}{\np{}{}{u_n^{+}}} \underset{n}{\rightharpoonup} y \mbox{ in } \HuzL,
\]
and by \eqref{compact_inclusion}, we can also assume that 
\begin{eqnarray}\label{convergence_y_n}
\left\{
\begin{array}{l}
y_n  \xrightarrow[n]{} y \mbox{ in } L^q(\Omega) \mbox{ for every } q \in \left[1,\frac{2N}{N-1}\right)\\
y_n(x)  \xrightarrow[n]{} y(x)\geq 0 \mbox{ for a.e. $x$ in }\Omega.
\end{array}
\right.
\end{eqnarray}

We distinguish two cases, according to whether $y\not \equiv 0$ or $y\equiv 0$.
In the former case we consider the set $Z = \{ x\in\Omega \,:\, y(x)=0\}$, whose complementary set $Z^c$ has positive measure. It is clear that
\[
u_n^+ \xrightarrow[n]{} \infty \mbox{ a.e. in } Z^c.
\]
Therefore, by {\bf Hii)}, we get
\begin{equation}\label{}
\frac{F(x,u_n^+)}{\np{}{2}{u_n^+}} = \frac{F(x,u_n^+)}{|u_n^+|^2} \frac{|u_n^+|^2}{\np{}{2}{u_n^+}} \xrightarrow[n]{} \infty \mbox{ a.e. in } Z^c.
\end{equation}

Now, combining {\bf Hi)} and {\bf Hii)}, we see that there exists $g\in L^1(\Omega)$ such that $\dfrac{F(x,u_n^+)}{\np{}{2}{u_n^+}} \geq g(x)$ for a.e. $x\in\Omega$, and we can use Fatou's lemma to obtain
\begin{equation}\label{Fatou_1}
\begin{aligned}
\liminf\int_{\Omega} \frac{F(x,u_n^+)}{\np{}{2}{u_n^+}}  
& \geq \int_{Z^c}\liminf\frac{F(x,u_n^+)}{\np{}{2}{u_n^+}} +  \int_{Z}\liminf\frac{F(x,u_n^+)}{\np{}{2}{u_n^+}}\\ & \geq 
C + \int_{Z^c}\liminf\frac{F(x,u_n^+)}{\np{}{2}{u_n^+}} =\infty,
\end{aligned}
\end{equation}
where $C$ is a constant.

On the other hand, \eqref{for_contradiction_2} implies that
\[
\left|\ds 1 - \frac{F(x,u_n^+)}{\np{}{2}{u_n^+}}\right| \leq \frac{M_3}{\np{}{2}{u_n^+}},
\] 
in contradiction with \eqref{Fatou_1}. Hence, in this case the claim is proved.

We now turn to the second case, i.e. $y\equiv 0$ in $\HuzL$.
Let us set $\gamma_n(t): = J(t u_n^+ )$, for $t\in[0,1]$. The sequence $t_n= \underset{t\in[0,1]}{\operatorname{argmax}}\, \gamma_n (t) $ is well defined, since $\gamma_n \in C[0,1]$.

For every $k\in (0,\|u_n^+\|)$ we set $\tilde{t}_n=\frac{k}{\np{}{}{u_n^+}}$, so that $\gamma_n(\tilde{t}_n) = J(k y_n)$ and $\tilde{t}_n \in (0,1)$; thus
\begin{equation}\label{Bo_1}
\begin{aligned}
J(t_n u_n^+) \geq J(\tilde{t}_n u_n^+) &= \frac{1}{2} \trin{}{2}{\tilde{t}_n u_n^+} - \frac{\la}{4} \int_{\Omega}\conv{G}{(\tilde{t}_n u_n^+)^2}|\tilde{t}_n u_n^+|^2 
- \int_{\Omega} F(x,\tilde{t}_n u_n^+) \\
&= \frac{1}{2}k^2 - \frac{\la}{4} \int_{\Omega}\conv{G}{(\tilde{t}_n u_n^+)^2}|\tilde{t}_n u_n^+|^2 - \int_{\Omega} F(x,\tilde{t}_n u_n^+).
\end{aligned}
\end{equation}

From \eqref{convergence_y_n} and {\bf Hi)}, we see that
\[
\int_{\Omega} F(x, \tilde{t}_n u_n^+) \xrightarrow[n\to \infty]{} 0.
\] 
In addition, by the Lebesgue Theorem and \eqref{maggconv}, we have that 
\[
\int_{\Omega}\conv{G}{(\tilde{t}_n u_n^+)^2}|\tilde{t}_n u_n^+|^2  \xrightarrow[n\to \infty]{} 0.
\]
Hence, from (\ref{Bo_1}),
given $M >0$, there exists $N=N(M)$ such that
\begin{equation}\label{stimettine}
\begin{aligned}
& \left|\int_{\Omega} F(x, \tilde{t}_n u_n^+)\right|\leq \frac{M}{8},\\
& \lambda\int_{\Omega}\conv{G}{(\tilde{t}_n u_n^+)^2}|\tilde{t}_n u_n^+|^2\leq \frac{M}{2}
\end{aligned}
\end{equation}
for all $n\geq N$. Choosing $k=\sqrt{M}$, from \eqref{Bo_1} and \eqref{stimettine}, we finally get
\[
J(t_n u_n^+) \geq \frac{M}{2} \mbox{ for every } n > {N},
\]
that is:
\begin{equation}\label{Bo_4}
J(t_n u_n^+) \xrightarrow[n\to \infty]{}\infty.
\end{equation}
The limit in \eqref{Bo_4} implies that $t_n\neq 0$ for $n$ large enough. On the other hand, we also have $t_n\neq1$. Indeed, if $t_n=1$, we would have
\begin{equation}\label{Bo_5}
J(u_n^+) = \frac{1}{2} \trin{}{2}{ u_n^+} - \frac{\la}{4} \int_{\Omega}\conv{G}{(u_n^+)^2}|u_n^+|^2 
- \int_{\Omega} F(x, u_n^+),
\end{equation}
which is bounded, thanks to \eqref{W_term_limited_5_W} and \eqref{for_contradiction_2}. 
Eventually, we conclude that $t_n \in (0,1)$ for $n$ large enough; this implies that
\begin{equation}\label{Bo_6}
\begin{aligned}
0 &= t_n \left. \frac{d}{d t}  J(t u_n^+)\right|_{t=t_n} = t_n \dual{}{J'(t_n u_n^+)}{u_n^+} = \dual{}{J'(t_n u_n^+)}{t_n u_n^+}\\
&=\trin{}{2}{t_n u_n^+} - \la t_n^4\int_{\Omega}\conv{G}{(u_n^+)^2}|u_n^{+}|^2  -\int_{\Omega} f(x,t_n u_n^{+})t_n u_n^{+} \\
&= \trin{}{2}{t_n u_n^+} - \la t_n^4\int_{\Omega}\conv{G}{(u_n^+)^2}|u_n^{+}|^2  -2 \int_{\Omega} F(x,t_n u_n^+)  - \int_{\Omega} \sigma(x, t_n u_n^+)
\end{aligned}
\end{equation}
Using hypothesis {\bf(Hiii)}, from \eqref{Bo_6} and \eqref{for_contradiction_3}, we get the existence of a positive constant $M_5$ such that
\begin{equation}\label{Bo_7}
\begin{aligned}
\trin{}{2}{t_n u_n^+} &- \la t_n^4\int_{\Omega}\conv{G}{(u_n^+)^2}|u_n^{+}|^2  -2 \int_{\Omega} F(x,t_n u_n^+)   \\&= \int_{\Omega} \sigma(x, t_n u_n^+)
\leq \int_{\Omega} \sigma(x, u_n^+) + |\beta^*|_1 \leq M_5,
\end{aligned}
\end{equation}
for every $n$ large enough.

Finally, we show that \eqref{Bo_4} implies that the left-hand-side of \eqref{Bo_7} diverges, thus obtaining a contradiction. Indeed, 
\[
2 J(t_n u_n^+) = \trin{}{2}{t_n u_n^+} - \frac{\la}{2} t_n^4 \int_{\Omega}\conv{G}{(u_n^+)^2}| u_n^+|^2 
- 2\int_{\Omega} F(x,\tilde{t}_n u_n^+) \xrightarrow[n\to \infty]{} \infty;
\]
but, by \eqref{W_term_limited_5_W}, we obtain the announced contradiction.
 
As a consequence, $(u_n^+)_n$ is bounded. From \eqref{C_with_un-}, we see that the whole sequence $(u_n)_n$ is bounded in $\HuzL$, as claimed.
\end{proof}

\begin{lem}
$(u_n)_n$ converges strongly in $\HuzL$.
\end{lem}
\begin{proof}
First, being $(u_n)_n$ bounded in $\HuzL$, up to subsequences, we may assume that there exists $u\in \HuzL$ such that
\begin{eqnarray}\label{conver_from_bounded}
\left\{
\begin{array}{l}
u_n \underset{n}{\rightharpoonup} u \mbox{ in } \HuzL, \\
u_n \xrightarrow[n]{} u \mbox{ in } L^q(\Omega) \mbox{ for every } q \in \left[1,\frac{2N}{N-1}\right)\\
u_n \xrightarrow[n]{} u \mbox{ a.e. in }\Omega.
\end{array}
\right.
\end{eqnarray}
We claim that $(u_n)_n$ converges strongly to $u$ in $\HuzL$.
In order to prove this claim we will exploit \eqref{Cerami_Condition}, and in particular the fact that
\[
J_+'(u_n) \xrightarrow[n\to \infty]{} 0 \mbox{ in }\big(\HuzL\big)^*,
\]
which implies that 
\begin{equation}\label{Cerami_2_second}
J_+'(u_n)(u_n - u) \xrightarrow[n\to \infty]{} 0.
\end{equation}

But
\begin{align}
& J_+'(u_n)(u_n - u)= \trin{}{2}{u_n}-\int_\Cil Du_n\cdot Du\,dxdx_{N+1}\label{final_strategy_1}\\
&- \la \int_{\Omega}\conv{G}{(u_n^+)^2} {u_n^+} (u_n - u)dx -\int_{\Omega} f(x,u_n^+)(u_n - u)dx.\label{final_strategy_2}
\end{align}
Then, showing that \eqref{final_strategy_2} goes to 0 as $n\to \infty$, \eqref{Cerami_2_second} and \eqref{final_strategy_1} imply that $u_n\to u$ in the Hilbert space $\HuzL$.
First, the convergence
\[
\int_{\Omega}\conv{G}{(u_n^+)^2} {u_n^+} (u_n-u)dx \xrightarrow[n\to \infty]{}0 
\]
follows directly from \ref{W_inequality}. Then, from {\bf Hi)},
\begin{equation}\label{f_conv_2}
\left| \int_{\Omega} f(x,u_n^+)(u_n-u)dx \right| \leq \int_{\Omega} a(x)|u_n-u|dx + c\int_{\Omega} |u_n^+|^{r-1} |u_n-u|dx,
\end{equation}
and from \eqref{conver_from_bounded} we have that all integrals in \eqref{f_conv_2} go to 0 as $n\to \infty$.

We have thus proved that $J_+$ satisfies the Cerami condition.
\end{proof}

\begin{proof}[Proof of Theorem \ref{theorem_2_Mountain_Pass_solutions}]
We apply the Mountain Pass Theorem obtaining the existence of a critical point $u_0 \in \HuzL$ for $J_+$, with $u_0 \neq 0$,
$J_+(u_0)>0$ and
\begin{equation}\label{MP_J_p_W}
J_+'(u_0) = 0
\end{equation}
Applying (\ref{MP_J_p_W}) to $u_0^-$ we see that
\begin{equation}\label{null_neg_part}
J_+'(u_0)( u_0^- )= \np{}{2}{u_0^-} = 0
\end{equation}
so that $u_0 \geq 0$, $J'(u_0) = J_+'(u_0) = 0$ and, consequently, $u_0$ is a weak nonnegative and non trivial solution of problem (\ref{equiv}).
Furthermore, the maximum principle implies that $u_0>0$ in $\Omega$, see \cite[Proposition 3.2]{mamu}.

In the same way, using the functional $J_-$, it is possible to obtain a solution $v_0<0$ in $\Omega$.

Finally, by Proposition \ref{Regularity_of_weak_solutions},  we get the bound on the solutions.
\end{proof}

\noindent
\textbf{Ackowledgement.} D.M. is member of the {\em Gruppo Nazionale per l'Analisi Mate\-ma\-tica, la Probabilit\`a e le loro Applicazioni} (GNAMPA) of the {\em Istituto Nazionale di Alta Matematica} (INdAM) and his research is supported by the National Research Project {\sl Variational and perturbative aspects of nonlinear differential problems}.

\bibliographystyle{amsplain}

\end{document}